\documentclass[12pt,leqno]{article}
\usepackage[utf8]{inputenc}
\usepackage{mathrsfs}
\usepackage{indentfirst,amsmath,amsthm,amsfonts,amsthm,amssymb}

\def\la{\lambda}

\newcommand\nlp[2]{\|#2\|_{L^{#1}(\Omega)}}
\newcommand\nlpr[2]{\|#2\|_{L^{#1}(\R^2)}}
\newcommand\nlpla[2]{\|#2\|_{L^{#1}(\Omega_\la)}}
\newcommand\nlpd[3]{\|#2\|_{L^{#1}(#3)}}
\newcommand\nlpdbig[3]{\Bigl\|#2\Bigr\|_{L^{#1}(#3)}}

\def\tu{\widetilde{u}}
\def\ep{\varepsilon}
\def\al{\alpha}
\def\om{\omega}

\def\Om{\Omega}
\def\Oml{\Omega_\la}
\def\HO{H_\Omega}

\def\H{H_\Om}
\def\HL{H_{\Oml}}
\def\hl{H_\lambda}
\def\th{\widetilde H}
\renewcommand\th{H_1}
\def\dt{\frac{\text{d}}{\text{d}t}}
\def\wu{\widetilde{u}}
\def\L2s{L^2_\sigma (\Omega)}
\def\Lps{L^p_\sigma (\Omega)}
\def\Lqs{L^q_\sigma (\Omega)}
\def\LpO{{L^p(\Omega)}}

\newtheorem{theorem}{Theorem}%[section]
\newtheorem{lemma}[theorem]{Lemma}%[section]
\newtheorem{proposition}[theorem]{Proposition}%[section]
\newtheorem{corollary}[theorem]{Corollary}%[section]

\numberwithin{equation}{section} 

\theoremstyle{remark}
\newtheorem{remark}[theorem]{Remark}

\newcommand{\real}{\mathbb{R}}

\newcommand{\be}{\begin{equation}}
\newcommand{\ee}{\end{equation}}
\renewcommand{\leq}{\leqslant}
\renewcommand{\geq}{\geqslant}

\def\R{\mathbb{R}}
\def\ppom{\mathbb{P}_{\Omega}}
\def\pp{\mathbb{P}}
\def\ppomla{\mathbb{P}_{\Omega_\la}}

\DeclareMathOperator\curl{curl}
\DeclareMathOperator\dive{div}
\DeclareMathOperator\supp{supp}

\title{\bf Self-similar asymptotics of solutions \\ to the Navier-Stokes system \\ in two dimensional exterior domain}

\author{Dragoş Iftimie, Grzegorz Karch \& Christophe  Lacave}
\def\adrese{
\begin{description}
\item[D. Iftimie:] Université de Lyon, CNRS, Université Lyon 1, Institut Camille Jordan, 43 bd. du 11 novembre, Villeurbanne Cedex F-69622, France.\\
Email: \texttt{iftimie@math.univ-lyon1.fr}\\
Web page: \texttt{http://math.univ-lyon1.fr/\~{}iftimie} 
\item[G. Karch:] Instytut Matematyczny, Uniwersytet Wroc{\l}awski,
 pl. Grunwaldzki 2/4, 50-384 Wroc{\l}aw, Poland.\\
Email: \texttt{grzegorz.karch@math.uni.wroc.pl}\\
Web page: \texttt{http://www.math.uni.wroc.pl/\~{}karch}
\item[C. Lacave:] Institut Mathématiques de Jussieu, Université Paris Diderot (Paris 7), 175 rue du Chevaleret, 75013 Paris, France.\\
Email: \texttt{lacave@math.jussieu.fr}\\
Web page: \texttt{http://people.math.jussieu.fr/\~{}lacave/} 
\end{description}
}
\begin{document} 
\maketitle
\begin{abstract}
We consider the 2D incompressible Navier-Stokes equations with Dirichlet boundary condition in the exterior of one obstacle. Assuming that the circulation at infinity of the velocity is sufficiently small, we prove that the large time behavior of the corresponding solution to the initial-boundary value problem is described by the Lamb-Oseen vortex. The later is the well-known explicit self-similar solution to the Navier-Stokes system in the whole space $\R^2$.
\end{abstract}

\section{Introduction}

It is well-known that the large time behavior of solutions of the initial-value problem for 
the Navier-Stokes equations considered either in the whole space $\R^n$, $n\geq 2$, or in an exterior domain
depends on integrability properties of initial conditions.
 In the finite energy case, that is when the velocity is square integrable, a solution tends to zero in $L^2(\R^n)$ as 
 time goes to infinity, see {\it e.g.}  \cite{MR1158939,MR775190, wiegner} and references therein.  
In this case, nonlinear effects are negligible for large values of time and asymptotics of solutions is determined by the 
corresponding Stokes semigroup.

On the other hand, 
when an initial velocity is not square integrable, a solution of the initial value problem for the
Navier-Stokes   in  $\R^n$ with $n\geq 2$ 
 is constructed in a so-called  scaling invariant space ({\it e.g.} in a homogeneous Besov space or
in a weak $L^n$-space) under suitable smallness assumption on initial conditions, see the
review article \cite{cannone} and the book \cite{MR1938147}. Here, the large time behavior of solutions is 
described by self-similar solutions to the Navier-Stokes system.

In this work, we contribute to the theory on the asymptotic behavior of solutions  of the Navier-Stokes system 
in a two dimensional exterior domain.
First, however, we recall that  the Navier-Stokes system in the whole space $\R^2$ has an explicit self-similar  solution called 
the Lamb-Oseen vortex 
\begin{equation}\label{Oseen}
\Theta(t,x)= \frac{x^\perp}{2\pi|x|^2}\Bigl(1-e^{-\frac{|x|^2}{4t}}\Bigr),\qquad \text{with} \quad  x^\perp =(x_2,-x_1),
\end{equation}
which appears in the large time expansions of other infinite energy solutions of this system. 
Let us explain this result.

For every  initial  vorticity $\omega_0 \in L^1(\R^2)$, one obtains the corresponding  divergence-free initial velocity field $u_0$ via the Biot-Savart law.
It is well-known that constructed-in-this-way  initial condition  belongs to the scaling invariant space $L^{2,\infty}(\R^2)$ 
(the weak $L^2$-space)
and the Navier-Stokes equations have a unique
global-in-time solution corresponding to such an initial datum, see \cite{MR1017289}.
Moreover, the large time behavior of solutions to the initial value problem for the 2D Navier-Stokes equations 
with  an initial vorticity from $L^1(\R^2)$  is given by the  multiple of the Lamb-Oseen vortex  $\alpha\Theta$,
 with  the {\it circulation at infinity}
$\al\equiv \int_{\R^2}\omega_0(x)\,dx$. 
This result was proved in \cite{GK88} if $\om_0$ is small in $L^1$, in \cite{Ca94}  in the case of small circulation, and in \cite{GW05} in  the general case.
In fact, due to the regularizing effect of the Navier-Stokes equations, as far as large time behavior is concerned,
an initial vorticity can be an arbitrary bounded Radon measure in $\R^2$, see  \cite{GG05}.

The aim of this paper is to show an analogous result on the large time behavior of solutions of the 2D Navier-Stokes equations in an exterior domain 
with the Dirichlet boundary condition, when the initial velocity is not square integrable. 
Here, however, due to the fact that a vorticity does not verify any reasonable boundary conditions, we cannot use the vorticity equation.
Hence, we formulate our hypothesis and results in terms of velocity rather than of vorticity.  
To see which hypothesis should be imposed  on an initial  velocity, 
we recall that,
for every bounded compactly supported vorticity,  one can construct  the corresponding velocity field in an exterior domain,
which
behaves when $|x|\to\infty$ as the vector field
 $x^\perp/|x|^2$, see \cite[Sec.~2.2]{MR1974460} and \cite[Sec.~3]{MR2244381} for more details. 
For this reason,  we assume in this work that our initial velocity is a small multiple of the particular vector  field  $x^\perp/|x|^2$ plus a large $L^2$ part.

Let us now be more precise. Assume that   $\Omega\subset \R^2$ 
is an exterior domain,
whose complement is a bounded, open, connected and  simply connected 
set, with a smooth boundary $\Gamma$.  Moreover, 
without loss of generality, we can assume that $B(0,1)\subset \R^2\setminus\Om$.
We consider the incompressible Navier-Stokes equations in $\Omega$ with the Dirichlet boundary condition
\begin{align}
\label{eq}&\partial_t u -\Delta u +u\cdot \nabla u +\nabla p =0,  
\qquad \dive u=0& &\text{for}\quad t>0, \quad x\in \Omega, \\
\label{boundary}&u(t,x)=0  & &\text{for}\quad t>0, \quad x\in \Gamma,\\
\label{ini}&u(0,x)=u_0(x) & & \text{for} \quad x\in \Omega.
\end{align}
Above, $u_0$ must be divergence free and tangent to the boundary. In the following, we assume that the initial condition is of the following particular form
\begin{equation}\label{decompu0}
u_0=\tu_0+\al\H  
\end{equation}
where $\wu_0\in \L2s$ is an arbitrary square integrable, divergence free, and tangent to the boundary vector field, and  
$\HO$ the unique harmonic vector field in $\Omega$ (\textit{i.e.} the unique vector field on $\Omega$ which is divergence free, curl free, vanishing at infinity, tangent to the boundary, and with circulation  equal to 1 on the boundary $\Gamma$). 
It was proved in \cite[Sec.~2.3]{MR1974460} that such a harmonic vector field $\H$ exists and behaves at infinity like $x^\perp/(2\pi|x|^2)$.  
Moreover, 
it follows directly from 
 \cite[Lemma 6 with $\ep=1$]{MR2244381} that every velocity field with a compactly supported bounded vorticity can be written 
under the form \eqref{decompu0}. 
Notice, however, that in an exterior domain, the circulation at infinity $\al$ is not the integral of the vorticity
as it is  in the full plane case. 
Namely, here, one has to subtract the circulation of the velocity on the boundary, so
the integral of the vorticity is in fact the total circulation of the velocity, see  \cite[Sec.~3.1]{MR1974460} for more details.

If the circulation at infinity  is sufficiently small, we are able to prove a counterpart of the result 
from \cite{GK88,Ca94, GW05}
on the large time
behavior of the Navier-Stokes in the whole plane. The following theorem contains the main result of this work.

\begin{theorem}\label{mainthm}
For every  $\wu_0\in\L2s$
there exists a constant $\al_0=\al_0(\tu_0,\Omega)>0$ such that 
for all  $|\al|\leq \al_0$
the solution of problem \eqref{eq}-\eqref{decompu0}
satisfies
\begin{equation} \label{asymp:p}
\lim_{t\to\infty}t^{\frac12-\frac1p}\|u(t)-\al\Theta (t)\|_\LpO=0
\end{equation}
for each $p\in (2,\infty)$.
\end{theorem}

In other words, Theorem \ref{mainthm} says that the large time behavior of solutions to the Navier-Stokes system in an exterior domain, supplemented with the Dirichlet boundary condition and particular initial condition \eqref{decompu0} is described by the explicit self-similar solution  \eqref{Oseen}  of the Navier-Stokes system.

\begin{remark}
The global-in-time well-posedness for problem \eqref{eq}-\eqref{decompu0} was established by
Kozono and Yamazaki \cite[Thm.4]{KY95}. The existence part of that
result requires an initial velocity $u_0$ to satisfy a smallness condition of the form
$\limsup\limits_{R\to\infty}R |\{x\in \Omega\;:\; |u_0(x)|>R\}|^{1/2}\ll 1$.
This condition is satisfied for every $\wu_0\in \L2s$.
Since $\H$ 
 is bounded, the $\limsup$ above is always zero in this case.
 % The uniqueness of solutions to  \eqref{eq}-\eqref{ini}
% holds true for divergence-free initial data in $L^{2,\infty}(\Om) + \L2s$  without any additional
% conditions, {\it cf.} \cite[Thm.~1]{KY95}.
\end{remark}

We apply the following strategy to prove Theorem \ref{mainthm}. In the next section, we prove
the limit relation \eqref{asymp:p} for
the linear evolution, that is when the nonlinear term $u\cdot\nabla u$ is skipped in equation \eqref{eq}. 
This is achieved by combining  results in \cite{MR2244381} with a rescaling technique used by  Carpio in \cite{Ca94}. 
Next,  in Section \ref{finalproof}, we show  that 
 we can assume, without loss of generality,
 that $u_0$ is small in the norm of the space $L^{2,\infty}(\Omega)$,
by replacing the initial condition in \eqref{ini}
with $u(t_0,x)$ with sufficiently large $t_0$ and  choosing sufficiently small $|\alpha|$ (see Lemma \ref{l4} below).
Finally, using the integral representation of solutions to problem \eqref{eq}-\eqref{ini}, we apply
  a stabilization argument inspired from \cite{BCK-BCP,MR2034160} to show that, for small data in 
$L^{2,\infty}_\sigma(\Omega)$, 
the asymptotic stability at the level of the Stokes equation implies the asymptotic stability at the level of the Navier-Stokes equations.

{\bf Notation.} In the following,  the space  $L^p_\sigma(\Omega)$ is the closure of the set of smooth, divergence-free, and compactly supported  
vector fields  $C^{\infty}_c(\Om)$
with respect to the usual $L^p$-norm.  We denote by $\ppom$ the Leray projection, \textit{i.e.} the $L^2$ orthogonal projection onto  $L^2_\sigma(\Omega)$,
which can be extended to a bounded operator on $L^p(\Omega)$ for every $p\in(1,\infty)$. Thus, the space  $L^p_\sigma(\Omega)$ 
is the image of   $L^p(\Omega)$ by $\ppom$. In a similar way, for every $p\in (1,\infty)$,
we define $L^{p,\infty}_\sigma(\Omega)=\ppom (L^{p,\infty}(\Omega))$, where $L^{p,\infty}(\Omega)$ is the Marcinkiewicz weak
$L^p$-space. Hence $u\in L^{p,\infty}_\sigma(\Omega)$ if $u\in L^{p,\infty}(\Omega)\times L^{p,\infty}(\Omega)$, $\dive u =0$ in $\Om$ and $u \cdot n=0$
on $\Gamma$, where $n$ is the normal vector to the boundary $\Gamma$.
The ball $B(0,R)\subset \R^2$ is centered at zero and of radius $R>0$.
By the letter $E$, we denote the extension operator of functions defined on $\Om$ to $\R^2$ with zero values outside the domain of definition.

%%%%%%%%%%%%%%%%%%%%%%%%%%%%%%%%%%%%%%%%%%%%%%%%%%
%%%%%%%%%%%%%%%%%%%%%%%%%%%%%%%%%%%%%%%%%%%%%%%%%%

 \section{Asymptotics of solutions to the linear evolution} \label{sec:linear}

%Let $\tu$ be solution of the Navier-Stokes equation in $\Om$ with initial data $\tu_0$. We set $w=u-\tu$.

It is well-known that the  Stokes operator 
associated with the following linear boundary value problem
\begin{align}
\label{eqS}&\partial_t v -\Delta v  +\nabla p =0, \qquad \dive v=0  & &\text{for}\quad t>0, \quad x\in \Om, \\
\label{boundaryS}&v(t,x)=0  & &\text{for}\quad t>0, \quad x\in \Gamma,\\
\label{iniS}&v(0,x)=v_0(x) & & \text{for} \quad x\in \Om,
\end{align}
 where $v_0$ is divergence free and tangent to the boundary, generates an analytic semigroup $S(t)$   on $\Lps$, for each $1<p<\infty$, see \cite{MR635201}. 
Moreover, this semigroup 
satisfies the  following decay $L^p$ estimates.

\begin{proposition} \label{ellpest}
Assume that  $1 < q <\infty$.

Let  $q \leq p\leq\infty$. There exists $K_1 = K_1(\Om,p,q)>0$ such that for every $v_0 \in \Lqs$
\begin{equation}\label{S1}
\|S(t)v_0\|_{L^p(\Omega)} \leq K_1 t^{\frac{1}{p} - \frac{1}{q}}\|v_0\|_{L^q(\Omega)}
\qquad \text{for all}\quad  t>0.
\end{equation} 
If,  in addition, we assume that $q<p\leq\infty$, then   for every $v_0\in L^{q,\infty}_\sigma(\Omega)$ we also have 
\begin{equation}\label{S1bis}
 \|S(t)v_0\|_{L^p(\Omega)} \leq K_1 t^{\frac{1}{p} - \frac{1}{q}}\|v_0\|_{L^{q,\infty}(\Omega)}
\qquad \text{for all}\quad  t>0.
\end{equation}

There exists $K_2 = K_2(\Om,q)>0$ such that 
for every $v_0\in L^{q,\infty}_\sigma(\Omega)$ we have 
the inequality
\begin{equation}\label{S1bis2}
 \|S(t)v_0\|_{L^{q,\infty}(\Omega)} \leq K_2 \|v_0\|_{L^{q,\infty}(\Omega)}
\qquad \text{for all}\quad  t>0.
\end{equation}

Let   $q \leq p \leq 2$. There exists $K_3 = K_3(\Om,p,q)>0$ such that 
\begin{equation}\label{S2}
 \|\nabla S(t)v_0\|_{L^p(\Omega)} \leq K_3 t^{-\frac{1}{2} + \frac{1}{p} - \frac{1}{q}}\|v_0\|_{L^q(\Omega)}
\qquad \text{for all}\quad  t>0.
\end{equation}

 Assume $q \geq 2$ and let $q \leq p < \infty$. Then there exists $K_4 = K_4(\Om,p,q)>0$
such that  for every matrix $F \in L^q(\Omega;M_{2\times2}(\real))$ 
\begin{equation}\label{S3}
 \| S(t)\ppom \dive F\|_{L^p(\Omega)} \leq K_4 t^{-\frac{1}{2} + \frac{1}{p} - \frac{1}{q}}
\|F\|_{L^q(\Omega)}
\qquad \text{for all}\quad  t>0,
\end{equation}
with the divergence $\dive$ computed  along rows of the matrix $F$.
\end{proposition}

 Estimates \eqref{S1}--\eqref{S2} were proved
in \cite{DS99a,DS99,KY95,MS97} and   estimate \eqref{S3} follows from \eqref{S2} by a duality argument
because the adjoint of $\nabla S(t)$ on  $\Lps$ is $S(t)\mathbb{P} \dive$.

\begin{remark}\label{rem:scal}
Recall the following  scale invariance of the Stokes equations: 
the vector $(v(t,x),p(t,x))$ is a solution of system \eqref{eqS} on $\Omega$ if and only 
if for every $\lambda>0$ the vector 
$(\lambda v(\lambda^2 t ,\lambda x),  \lambda^2 p(\lambda^2 t ,\lambda x))$ is 
a solution of the same system on $ \Om/\la=\{x\in\R^2 \ ;\ \la x\in\Om\}$.
It follows from this scale invariance that the constants $K_1$,\dots,$K_4$ associated to $ \Om/\la$ are 
independent of $\lambda$. 
\end{remark}

The following corollary contains a minor improvement of the decay estimate \eqref{S1}.

\begin{corollary}\label{cor:dec}
Assume that  $1 < q <\infty$ and let $v_0 \in \Lqs$. Then for every  $p\in (q,\infty)$
\begin{equation*}
\lim_{t\to\infty} t^{ \frac{1}{q}-\frac{1}{p}} \|S(t)v_0\|_{L^p(\Omega)}=0. 
\end{equation*} 
\end{corollary}
\begin{proof}
This limit relation is clear when the initial datum is smooth and compactly supported.
To show it for all   $v_0 \in \Lps$, it suffices to use a standard density argument combined with estimate \eqref{S1}.
\end{proof}

Now, we consider the linear problem \eqref{eqS}-\eqref{iniS}
 with the initial datum $v_0=\HO$,  where $\HO$ the unique harmonic vector field in $\Omega$.
The main goal of this section is to show that the large time behavior of $S(t)\HO$ is described by the Lamb-Oseen vortex $\Theta$. More precisely, we will prove the following theorem.

\begin{theorem}\label{thm:oseen}
For every $p\in (2,\infty)$, we have
$
\lim\limits_{t\to\infty}t^{\frac12-\frac1p}\nlp p{S(t)\HO-\Theta(t)}=0.
$
\end{theorem}

The reminder of this section is devoted to the proof of this theorem. 
Here, we use a scaling argument that was also applied  in \cite{Ca94} 
to study large time asymptotics for the Navier-Stokes equations. 
Hence, for every  $\la\geq1$, we  define 
$$
\Omega_\la\equiv \Om/\la=\{x\in\R^2 \ ;\ \la x\in\Om\}.
$$ 
The  vector field $\la\H(\la x)$ is divergence free, curl free, tangent to the boundary of $\Oml$, vanishes at infinity and 
has circulation equal to 1 on $\partial\Oml$. 
Thus, by  \cite[Prop.~2.1]{MR1974460}, this rescaled vector field  has to  
 be equal to the unique harmonic vector field on $\Oml$, namely, we have the identity 
\begin{equation}\label{HL}
\HL(x)= \la\H(\la x).
\end{equation}
Let us now denote by $S_\la(t)$ the Stokes semi-group on the domain $\Oml$ and let us define
\begin{equation}\label{HSL}
\hl(t,x)\equiv S_\la(t)\HL.  
\end{equation}
By the scaling invariance of equations \eqref{eqS}, by  \eqref{HL}, and by the uniqueness of solutions to the Stokes problem, 
we infer that
\begin{equation*}
\hl(t,x)=\la H_1(\la^2t,\la x),  
\end{equation*}
where we put $H_1(t,x)=S(t)\H$. Recalling, moreover,  the scaling property of the Lamb-Oseen vortex 
 $\la \Theta(\la^2t,\la x)=\Theta(t,x)$, we observe that the conclusion of Theorem \ref{thm:oseen} is equivalent to
\begin{equation*}
\lim_{\la\to\infty}\nlpla p{\hl(1)-\Theta(1)}=0 \qquad \text{for every} \quad p \in (2,\infty).
\end{equation*}
In the following, we
 denote by $E$ the extension operator to $\R^2$ with zero values outside the domain of definition. 
Since $\Theta(1)$ is a bounded function,  we immediately obtain that
$\lim\limits_{\lambda\to\infty} \|\Theta(1)\|_{L^p(\R^2\setminus \Om_\la)}=0$. Hence, 
in order to prove Theorem \ref{thm:oseen}, it suffices to show that
\begin{equation}\label{EHstrong}
E\hl (1,x)\stackrel{\la\to\infty}{\longrightarrow} \Theta(1,x) \quad \text{strongly in} \quad  L^p(\R^2) \qquad \text{for every} \quad p\in(2,\infty).
\end{equation}

First, we state a result on the weak convergence.

\begin{lemma}
Let $\hl(t,x)=\la H_1(\la^2t,\la x)$. Then
\begin{equation}\label{EHweak}
E\hl (1,x)\stackrel{\la\to\infty}{\longrightarrow} \Theta(1,x) \quad \text{weakly in} \quad   L^p(\R^2)\qquad \text{for every} \quad p\in(2,\infty).
\end{equation}
\end{lemma}

\begin{proof}%[Sketch of the proof.]
Observe now that due to the identity  $\|\HL\|_{L^{2,\infty}_\sigma(\Omega_\la)}=\|\H\|_{L^{2,\infty}_\sigma(\Omega)}$ 
for every $\la\geq1$, 
the scaling invariant  estimate \eqref{S1bis} 
 implies that the family $\{E\hl(1)\}_{\lambda\geq1}$ is bounded in $L^p(\R^2)$
for every $p\in (2,\infty)$, hence weakly compact in these spaces.

From now on, we follow the reasoning from  \cite{MR2244381}, 
where the authors considered the Navier-Stokes equations in $\Oml$ with a more general initial velocity. In our case, the initial vorticity vanishes while in  \cite{MR2244381} the vorticity is smooth, independent of $\la$ and compactly supported in $\R^2\setminus\{0\}$. The difference between the Stokes and the Navier-Stokes equations is the bilinear term $u\cdot\nabla u$ which only complicates matters. Therefore, ignoring all additional difficulties caused  by the bilinear term, the results proved in \cite{MR2244381} go through to our case. Note that the smallness assumption required in \cite{MR2244381} is irrelevant in this work  since we deal with a linear equation.

Let us be more precise.
It was proved in  \cite{MR2244381} (see Proposition 18 and the end of the proof of Theorem 22) that $\pp_{\R^2}[\eta^\la E\hl]$ converges to the Lamb-Oseen vortex $\Theta$ when $\la\to\infty$, up to a subsequence, uniformly in time with values in $H^{-3}_{loc}(\R^2)$. The precise definition of the cut-off function $\eta^\la$ is not required here (the interested reader can find it in relation (4.1) of  \cite{MR2244381} with $\ep=1/\la$). We only need to know that $0\leq\eta^\la\leq1$, that $\eta^\la$ vanishes in the neighborhood of the boundary of $\Oml$ and  that $\eta^\la(x)\equiv1$ for all $|x|>C/\la$. 

In particular, we have that $\pp_{\R^2}[\eta^\la E\hl(1)]\to\Theta(1)$ in  $H^{-3}_{loc}(\R^2)$  when $\la\to\infty$, up to a subsequence. On the other hand, the sequence $\pp_{\R^2}[\eta^\la E\hl(1)]$ is bounded in $L^p(\R^2)$ since $\hl(1)$ is bounded in $L^p(\Oml)$. By uniqueness of limits, we infer that $\pp_{\R^2}[\eta^\la E\hl(1)]\to\Theta(1)$ weakly in $L^p(\R^2)$ as $\la\to\infty$.

Finally, we observe that
\begin{multline*}
\nlpr p{ \pp_{\R^2}[\eta^\la E\hl(1)]-E\hl(1)}
= \nlpr p{ \pp_{\R^2}[(\eta^\la-1) E\hl(1)]}\\
\leq C \nlpr p{(\eta^\la-1) E\hl(1)}
\leq C\nlp\infty{\hl(1)}\operatorname{mes}(B(0,C/\la))^{\frac1p}
\leq C\la^{-\frac2p}\stackrel{\la\to\infty}{\longrightarrow}0.
\end{multline*}
This completes the proof of the lemma.
\end{proof}

Consequently,  to prove the strong convergence \eqref{EHstrong}, in view of the weak convergence \eqref{EHweak}, 
it suffices to show that $\{E\hl(1)\}_{\lambda\geq1}$ is relatively compact in  $L^p(\R^2)$ for every $p\in(2,\infty)$. 
Here, we proceed in two steps;  we show 
that  the family $\{E\hl(1)\}_{\lambda\geq1}$   is:
\begin{itemize} 
\item[i)] relatively compact in $L^p_{loc}(\R^2)$ for every  $p\in(2,\infty)$ (Lemma \ref{lem:comp1}, below), 
\item[ii)] small in the  $L^p$-sense for large $|x|$, uniformly in $\la\geq1$ (Lemma \ref{lem2}). 
\end{itemize}

Then, the relative compactness of the family $\{E\hl(1)\}_{\lambda\geq1}$ in the space $L^p(\R^2)$ is a consequence of a standard diagonal argument.
Here, a set is called to be relatively compact in  $L^p_{loc}(\R^2)$ if it is relatively compact in  $L^p (B(0,R))$ for every $R>0$.

In the following two lemmas,  $R>1$ is a sufficiently large constant  
and  
\begin{equation}\label{hR}
h_R(x)=h(x/R), \qquad  \text{where}\quad h\in C^\infty(\R^2)
\end{equation}
is such that $h(x)=0$ for $|x|<1$ and $h(x)=1$ for $|x|>2$. 

\begin{lemma} \label{lem:comp1}
Let $\hl(t)$ be defined in \eqref{HSL}.
 The set  $\{E\hl(1)\}_{\lambda\geq1}$ is relatively compact in  $L^p_{loc}(\R^2)$ for every $p\in(2,\infty)$. 
\end{lemma}

\begin{proof}
Here, as the usual practice, one could  show  $L^p$-estimates 
for $\nabla E\hl(1)$ which are uniform in $\lambda\geq1$. 
Unfortunately, we do not know any scaling invariant gradient estimate for solutions of the Stokes equation with
initial conditions from $L^{2,\infty}(\Omega)$. Thus, we have to  proceed in a different manner. 

Recall first that, by \cite[Prop.~2.1]{MR1974460}, 
 the vector field $\H$ is smooth, bounded, and 
 there is a constant $C>0$ such that  $|\H(x) | \leq C/|x|$ for all $x\in\Omega$ 
(recall that $\Omega \subset \R^2\setminus B(0,1)$). 
Since the rescaled harmonic vector field $\HL$ is divergence free and tangent to the boundary, we can write  the following decomposition
\begin{equation*}
\HL=\ppomla\HL=\ppomla(h_R\HL)+\ppomla[(1-h_R)\HL],
\end{equation*}
with the cut-off function $h_R$ defined in \eqref{hR}.

Obviously, the Leray projector $\ppomla$ is a bounded operator on  the space
$L^q(\Oml)$ for each  $1<q<\infty$, with norm independent of $\la$. 
Thus, for fixed  $q\in(1,2)$, using the identity \eqref{HL} we estimate
\begin{equation}\label{boundlq}
\begin{split}
\nlpla q{\ppomla[(1-h_R)\HL]}&\leq C \nlpla q{(1-h_R)\HL} \leq C\la \nlpd q{\H(\la \cdot)}{\Oml\cap B(0,2R)} \\
&=C\la^{1-\frac2q} \nlpd q{\H}{\Om\cap B(0,2R\la)}
\leq  C\la^{1-\frac2q} \nlpdbig q{\frac1{|\cdot|}}{1<|x|<2R\la}\\
&\leq C(q, \Omega) R^{\frac2q-1}.
\end{split}
\end{equation}
Therefore, the quantity
 $\ppomla[(1-h_R)\HL]$ is bounded in $L^q(\Oml)$ with $q\in (1,2)$, uniformly with respect to $\la$. Now, we deduce from the 
scaling invariant decay 
estimates \eqref{S1} and  \eqref{S2} that $\big\{S_\la(1)\ppomla[(1-h_R)\HL]\big\}_{\lambda\geq1}$ is bounded in $H^1({\Om}_\la)$. Moreover, since  $S_\la(1)\ppomla[(1-h_R)\HL]$ vanishes on the boundary of $\Oml$ we have the relation
$$
E\nabla S_\la(1)\ppomla[(1-h_R)\HL]=\nabla ES_\la(1)\ppomla[(1-h_R)\HL]
$$ 
which implies that   $\big\{ES_\la(1)\ppomla[(1-h_R)\HL]\big\}_{\lambda\geq1}$ is bounded in $H^1(\R^2)$. 
By the compactness of the  Sobolev imbedding $H^1(\R^2)\subset L^p_{loc}(\R^2)$, 
we infer that the set $ \big\{ES_\la(1)\ppomla[(1-h_R)\HL]\big\}_{\lambda\geq1}$ is relatively compact in  $L^p_{loc}(\R^2)$ for all $p\in(2,\infty)$. 

On the other hand,   calculations similar to those in \eqref{boundlq} with  $p\in(2,\infty)$ lead to the inequality 
\begin{equation*}
\nlpla p{\ppomla h_R\HL]}\leq  C\la^{1-\frac2p} \nlpdbig p{\frac1{|\cdot|}}{|x|>R\la}\leq C(p, \Omega) R^{\frac2p-1}.  
\end{equation*}
Using the decay estimate \eqref{S1} we infer that
\begin{equation}\label{ESLp}
\begin{split}
\nlpr p {ES_\la(1)\ppomla(h_R\HL)}&=\nlpla  p {S_\la(1)\ppomla(h_R\HL)} \\
&\leq C \nlpla  p {\ppomla(h_R\HL)} \leq C(p, \Omega) R^{\frac2p-1}.  
\end{split}
\end{equation}

Finally, since $ES_\la(1)\ppomla(h_R\HL)$ tends to zero  in $L^p(\R^2)$ as $R\to\infty$ uniformly in $\la$ and  since 
the family 
$ \{ES_\la(1)\ppomla[(1-h_R)\HL]\}_{\lambda\geq1}$ is relatively compact in  $L^p_{loc}(\R^2)$ for every fixed $R$, we infer that 
\begin{equation}\label{EHl}
E\hl(1)=ES_\la(1)\HL=ES_\la(1)\ppomla(h_R\HL)+ES_\la(1)\ppomla[(1-h_R)\HL]
\end{equation}
is relatively compact in  $L^p_{loc}(\R^2)$. 
\end{proof}

\begin{lemma}\label{lem2}
Let $\hl(t)$ be defined in \eqref{HSL} and $h_R$ be defined in \eqref{hR}.
For every $p\in (2,\infty)$, 
 we have that $\lim\limits_{R\to\infty}\|h_RE\hl(1)\|_{L^p(\R^2)}=0$  uniformly in $\la\geq 1$.
\end{lemma}

\begin{proof}
Let $\ep>0$ be an arbitrary small constant and $R_0=R_0(\ep)$ be a large constant to be chosen later. 
We estimate the $L^p$-norm of $h_R E\hl(1)$ using the decomposition of $ E\hl(1)$ from \eqref{EHl} with $R=R_0$.
%	
%	
%	\begin{equation*}
%	E\hl(1)=ES_\la(1)\HL=ES_\la(1)\ppomla(h_{R_0}\HL)+ES_\la(1)\ppomla[(1-h_{R_0})\HL].  
%	\end{equation*}
First, repeating the calculations from \eqref{ESLp} 
we have 
\begin{equation*}
\nlpr p {h_RES_\la(1)\ppomla(h_{R_0}\HL)}\leq  C(p, \Omega) R_0^{\frac2p-1}.    
\end{equation*}
Since the right-hand side tends  to 0 as $R_0\to\infty$ uniformly in $\la\geq 1$, there exists $R_0$ independent of $\la$ such that
\begin{equation*}
\nlpr p {h_RES_\la(1)\ppomla(h_{R_0}\HL)}\leq \ep\qquad\text{for all }\la\geq 1.  
\end{equation*}

Now, for fixed  $R_0$,  we  show  that
\begin{equation*}
\lim_{R\to\infty}h_REv_\la(1)=0\qquad\text{with}\quad v_\la(t)=  S_\la(t)\ppomla[(1-h_{R_0})\HL]
\end{equation*}
where the convergence is in the norm of $L^p(\R^2)$ and is uniform with respect to $\la\geq 1$.

First, it follows from relation \eqref{boundlq} that $v_\la(0)=\ppomla[(1-h_{R_0})\HL]$ is bounded in $L^q(\Oml)$ for each $1<q<2$, 
uniformly in $\la\geq 1$. Using the decay estimates for the Stokes equation stated in \eqref{S1} and \eqref{S2}, 
we infer that $v_\la$ verifies
\begin{equation}\label{estvla}
\nlpla q{v_\la(t)}\leq C(q) \quad\text{and}\quad    \nlpla 2{\nabla v_\la(t)}\leq C(\eta)t^{-\eta},\qquad\text{uniformly in} \quad \la\geq 1,
\end{equation}
for each $q\in (1,2)$, $\eta=1/q\in (1/2,1)$, and all $t>0$.

Let $\om_\la$ denote the curl of $v_\la$. The quantity $Eh_R\om_\la$ verifies the following equation in the full plane
\begin{equation*}
%  \label{eqomla}
\partial_t( Eh_R\om_\la)-\Delta( Eh_R\om_\la)=E\bigl[\Delta h_R\om_\la-2\dive(\nabla h_R\om_\la)\bigr]
\qquad\text{in}\quad \R_+\times\R^2,
\end{equation*}
supplemented with the zero initial datum, because 
\begin{equation*}
\begin{split}
Eh_R\om_\la(0)&= Eh_R\curl v_\la(0)=Eh_R\curl\ppomla[(1-h_{R_0})\HL]\\
&=Eh_R\curl[(1-h_{R_0})\HL]=-Eh_R\HL\cdot\nabla^\perp h_{R_0}=0
\end{split}
\end{equation*}
for $R>2R_0$. 
In these calculations, we used the fact that, for any vector field $w$, the quantity $w-\ppomla w$ is a gradient, that $\HL$ is curl free, that $\supp h_R\subset\{|x|>R\}$ and that $\supp \nabla h_{R_0}\subset\{R_0<|x|<2R_0\}$. 

The Duhamel principle for the inhomogeneous heat equation in the full plane implies now that
\begin{equation}\label{EhR}
 Eh_R\om_\la(1)=\int_0^1 \frac1{4\pi(1-s)}e^{-\frac{|\cdot|^2}{4(1-s)}}\ast E\bigl[\Delta h_R\om_\la-2\dive(\nabla h_R\om_\la)\bigr](s)\,ds. 
\end{equation}
Let  $q\in (1,2)$ satisfy $1/q=1/2+1/p$. We estimate  the $L^q$-norm of $Eh_R\om_\la(1)$
using  relation  \eqref{EhR}  in the following way
\begin{align*}
\nlpr q{ Eh_R\om_\la(1)}
&\leq C\int_0^1  \frac1{1-s}\nlpr 1{e^{-\frac{|\cdot|^2}{4(1-s)}}}\nlpr q{\Delta h_R\om_\la}\,ds\\
&\hskip 4cm +C\int_0^1  \frac1{1-s}\nlpr 1{\nabla \bigl[e^{-\frac{|\cdot|^2}{4(1-s)}}\bigr]}\nlpr q{\nabla h_R\om_\la}\,ds\\
&\leq C\int_0^1\nlpr p{\Delta h_R}\nlpla2{\om_\la}+C\int_0^1\frac1{\sqrt{1-s}}\nlpr p{\nabla h_R}\nlpla2{\om_\la}\,ds\\
&\leq C\int_0^1(\nlpr p{\Delta h_R}+\nlpr p{\nabla h_R})(1+(1-s)^{-\frac12})s^{-\frac34}\,ds\\
&\leq CR^{\frac2p-1},
\end{align*}
where we used \eqref{estvla}. We conclude, using again \eqref{estvla}, that
\begin{align*}
\nlpr q{\curl(Eh_Rv_\la(1))}
&\leq \nlpr q{Eh_R\om_\la(1)}+\nlpr q{Ev_\la(1)\cdot\nabla^\perp h_R}\\
&\leq   \nlpr q{Eh_R\om_\la(1)}+\nlpla q{v_\la(1)} \nlpr\infty{\nabla h_R} \\
&\leq CR^{\frac2p-1}.
\end{align*}
On the other hand, we can also bound
\begin{equation*}
  \nlpr q{\dive(Eh_Rv_\la(1))}=  \nlpr q{Ev_\la(1)\cdot\nabla h_R}\leq \nlpla q{v_\la(1)}  \nlpr\infty{\nabla h_R}\leq \frac CR.
\end{equation*}
Finally, putting together these estimates, we obtain 
\begin{equation*}
\begin{split}
\nlpr p{Eh_Rv_\la(1)}&\leq C \nlpr q{\nabla(Eh_Rv_\la(1))}\\
&\leq  C\nlpr q{\dive(Eh_Rv_\la(1))}+C  \nlpr q{\curl(Eh_Rv_\la(1))}\leq CR^{\frac2p-1} \stackrel{R\to\infty}{\longrightarrow}0
\end{split}
\end{equation*}
uniformly in $\la\geq 1$. This completes the proof of Lemma  \ref{lem2}.
\end{proof}

%%%%%%%%%%%%%%%%%%%%%%%%%%%%

\section{Proof of the main result}
\label{finalproof}

The proof of Theorem \ref{mainthm}  proceeds  in two steps. 
First, we reduce the problem to the  study of initial velocities, which are small in the $L^{2,\infty}$-norm. 
In the second step, we assume that  $u_0$ is sufficiently small in $L^{2,\infty}(\Omega)$ and 
we show that if  the solution of the Stokes problem \eqref{eqS}-\eqref{iniS} converges towards  the Lamb-Oseen vortex, 
then so does the solution of the  nonlinear problem. 
Once these two steps are completed, Theorem \ref{mainthm} follows from Theorem \ref{thm:oseen}.

\subsection{Reduction to the case of small initial velocity.}

We begin by recalling a classical result on the $L^2$-decay of weak solutions to problem \eqref{eq}-\eqref{ini}.
\begin{theorem}[{Borchers \& Miyakawa  \cite[Thm.~1.2]{MR1158939}}]\label{thm:L2decay}
 For every $\wu_0\in \L2s$ there is a unique weak solution 
$\tu\in L^\infty((0,\infty);L^2(\Om))\cap L^2_{loc}([0,\infty);H^1(\Om))$, 
of  problem \eqref{eq}-\eqref{ini} with $u_0=\tu_0$ as an initial datum, 
such that
$
  \lim\limits_{t\to\infty}\|\tu(t)\|_{L^2}=0.
$
\end{theorem}

We show now the following auxiliary result.

\begin{lemma}\label{l4}
Let $u$ be a solution to \eqref{eq}-\eqref{ini} with $u_0$ of the form \eqref{decompu0} with arbitrary $\wu_0\in\L2s$ and $\alpha\in \R$.
Denote by $\tu$ the weak solution from Theorem \ref{thm:L2decay}.
For every $t_0>0$, we have that
\begin{equation*}
\sup_{[0,t_0]}\nlp2{u(t)-\tu(t)-\alpha S(t)\H}\to 0 \qquad \text{as}\quad \alpha \to 0.
  \end{equation*}
\end{lemma}

\begin{proof}
We show a $L^2$-estimate for the function $z(t)\equiv u(t)-\tu(t)-\alpha S(t)\H$ which satisfies the following equation
\begin{equation}\label{eqz}
\partial_t z-\Delta z+(\tu+z+\al\th)\cdot\nabla(\tu+z+\al\th)-\tu\cdot\nabla\tu+\nabla \overline p=0,
\end{equation}
where
$%\begin{equation*}
  \th(t)=S(t)\H.
$%\end{equation*}

We multiply equation \eqref{eqz} by $z$ and integrate in the space variable  to obtain, after some integrations by parts, 
\begin{equation}\label{zl2}
\begin{split}
\frac12\dt\nlp2z^2+\nlp2{\nabla z}^2  
=&\al\int \tu\cdot\nabla z\cdot\th-\int z\cdot\nabla\tu\cdot z+\al\int z\cdot\nabla z\cdot\th\\
&+\al\int\th\cdot\nabla z\cdot\tu +\al^2\int\th\cdot\nabla z\cdot\th\\
% \leq & |\al |  \int |\tu | \,  |\nabla z| \,  |\th| + \int |z| \, |\nabla \tu | \,  | z | 
% + |\al | \int | z | \, | \nabla z | \,  | \th |  \\
% &+|\al | \int |\th | \, | \nabla z | \,  | \tu | + \al^2\int | \th | \, | \nabla z | \, | \th |\\
\equiv& I_1+I_2+I_3+I_4+I_5.
\end{split}
\end{equation}
Using the following interpolation inequality
\begin{equation*}
\nlp4f\leq C\nlp2f^{\frac12}\nlp2{\nabla f}^{\frac12}\qquad \text{for every}\quad f\in H^1_0(\Om),
\end{equation*}
we  bound each term on the right-hand side of \eqref{zl2} in the following way
\begin{align*}
I_1&\leq| \al |\nlp2\tu\nlp2{\nabla z}\nlp\infty\th  
\leq\frac16\nlp2{\nabla z}^2+C\al^2\nlp2\tu^2\nlp\infty\th^2,\\
I_2&\leq \nlp4z^2\nlp2{\nabla\tu}\leq C\nlp2z\nlp2{\nabla z}\nlp2{\nabla\tu} \leq\frac16\nlp2{\nabla z}^2+C\nlp2z^2\nlp2{\nabla \tu}^2,\\
I_3&\leq| \al |\nlp2z\nlp2{\nabla z}\nlp\infty\th\leq\frac16\nlp2{\nabla z}^2+C\al^2\nlp2z^2\nlp\infty\th^2,\\
I_4&\leq| \al |\nlp\infty\th\nlp2{\nabla z}\nlp2\tu\leq\frac16\nlp2{\nabla z}^2+C\al^2\nlp\infty\th^2\nlp2\tu^2,\\
I_5&\leq\al^2\nlp4\th^2\nlp2{\nabla z}\leq\frac16\nlp2{\nabla z}^2+C\al^4\nlp4\th^4.
\end{align*}
Plugging the above inequalities  into \eqref{zl2} yields
\begin{equation*}
\begin{split}
  \dt\nlp2z^2+\frac13\nlp2{\nabla z}^2
\leq& C\nlp2z^2(\nlp2{\nabla\tu}^2+\al^2\nlp\infty\th^2)\\
&+C\al^2\nlp\infty\th^2\nlp2\tu^2 +C\al^4\nlp4\th^4.
\end{split}
\end{equation*}
Recall that  $z_0=0$ and $H_1(t)=S(t)H_\Om$. Thus, the Gronwall inequality implies
\begin{equation*}
\begin{split}
\sup_{[0,t_0]}\nlp2z^2\leq& C\al^2\Bigl(\int_0^{t_0} \|S(\tau)H_\Omega\|^2_{L^\infty(\Omega)}
   \nlp2{\tu(\tau)}^2\,d\tau +\al^2\int_0^{t_0}\nlp4{S(\tau)H_\Omega}^4\,d\tau\Bigr)\\
&\times\exp\Bigl(C\int_0^{t_0}\nlp2{\nabla\tu(\tau)}^2 \,d\tau+C\al^2\int_0^{t_0}\nlp\infty{S(\tau)H_\Omega}^2\,d\tau\Bigr). 
\end{split}
\end{equation*}
Since $\tu\in L^\infty((0,t_0);L^2(\Om))\cap L^2((0,t_0);H^1(\Om))$ and since $\H\in L^p(\Om)$ for all $p\in (2,\infty]$, 
we infer from the decay estimate \eqref{S1} that the right-hand side of the above inequality  is finite and tends to zero 
as $\al\to 0$. This completes the proof of  Lemma \ref{l4}.
\end{proof}

In the following, we need a simple consequence of this lemma.

\begin{corollary}\label{corol}
Under the assumptions of Lemma \ref{l4},
for every $\ep>0$, there exists $\al_0=\al_0(\Omega, \tu_0,\ep)>0$ and $T_0=T_0(\Omega, \tu_0,\ep)\geq 0$ such that if $|\al|\leq \al_0$ then $\|u(T_0)\|_{L^{2,\infty}(\Omega)}\leq\ep.$  
\end{corollary}
\begin{proof}
Let $\ep>0$ be arbitrary.
First, by Theorem \ref{thm:L2decay},  we choose $T_0$ so large to have  $\nlp2{\tu(T_0)}\leq \ep/3$. 
Next, by \eqref{S1bis2}, we have the following bound
\begin{equation*}
\|\al S(t)\H\|_{L^{2,\infty}(\Omega)}\leq K_2|\al|  \|\H\|_{L^{2,\infty}(\Omega)}\leq \frac\ep3,
\end{equation*}
provided that
$ %\begin{equation*}
\al_0\leq  \ep/\big(3K_2 \|\H\|_{L^{2,\infty}(\Omega)}\big). 
$ % \end{equation*}
Finally,  we infer from Lemma \ref{l4} that if $\al_0$ is sufficiently small,  then
\begin{equation*}
\sup_{[0,T_0]}\nlp2{u(t)-\tu(t)-\al S(t)\H}\leq \frac\ep3.  
\end{equation*}
Consequently,  %$\|u(T_0)\|_{L^{2,\infty}(\Omega)}\leq \ep$.
\begin{equation*}
\begin{split}
\|u(T_0)\|_{L^{2,\infty}(\Omega)}\leq &  \|u(T_0)-\tu(T_0)-\al S(T_0)\H\|_{L^{2,\infty}(\Omega)}+\|\tu(T_0)\|_{L^{2,\infty}(\Omega)}\\
&+\|\al S(T_0)\H\|_{L^{2,\infty}(\Omega)}\leq\ep.
\end{split}
\end{equation*}
\end{proof}

\subsection{Large time asymptotics for small velocities.}

Now, we  show that, for sufficiently small initial conditions,
 if the linear evolution converges to the Lamb-Oseen vortex, then so does the nonlinear evolution. 
This result is stated in the following proposition.

\begin{proposition}\label{reduction-linear}
Let $u_0\in L^{2,\infty}_\sigma(\Omega)$ and denote by  $u=u(t,x)$ the corresponding solution to \eqref{eq}--\eqref{ini}. 
There exists $\ep=\ep(\Omega)>0$ such that if $\max\{\|u_0\|_{L^{2,\infty}(\Omega)}, |\al|\}\leq\ep$ and if
\begin{equation}  \label{limsupx}
\lim_{t\to\infty}t^{\frac12-\frac1p}  \nlp{p}{S(t)u_0-\al \Theta(t)}=0 \quad \text{for every}\quad p\in (2,\infty)
\end{equation}
then 
\begin{equation}
  \label{limsup}
\lim_{t\to\infty}t^{\frac12-\frac1p}  \nlp{p}{u(t)-\al \Theta(t)}=0 \quad \text{for every}\quad p\in (2,\infty).
\end{equation}
\end{proposition}
\begin{proof}
It follows  from the results in \cite[Thm.~3]{KY95}  that
for every  $p\in (2,\infty) $ there exists a constant $C(p)>0$ such that 
\begin{equation}\label{KY:dec}
\sup_{t>0}t^{\frac12-\frac1p}  \nlp{p}{u(t)}\leq C(p)\ep,
\end{equation}
provided $\ep>0$ is sufficiently small.

First, we show relation \eqref{limsup} for $p=4$. The Duhamel principle  allows to  rewrite problem \eqref{eq}-\eqref{ini}
as the integral equation
\begin{equation*}
u(t)= S(t)u_0-\int_0^t S(t-s)\ppom\dive(u\otimes u)(s)\, ds.
\end{equation*}
Subtracting  the Lamb-Oseen vortex $\Theta$ on the both sides of the above relation we get
\begin{align*}
 u(t)-\al\Theta(t)&= S(t)u_0-\al\Theta(t)-\int_0^t S(t-s)\ppom\dive(u\otimes u)(s)\, ds\\
&= S(t)u_0-\al\Theta(t)-\int_0^t S(t-s)\ppom\dive\big[(u-\al\Theta)\otimes u+\al\Theta\otimes(u-\al \Theta)\big](s)\, ds,
\end{align*}
because  $\ppom\dive(\Theta\otimes \Theta)=0$. This is  a consequence of the fact that  the vector field 
 $\Theta$ is  orthogonal to the gradient  of a  radial function so that $\dive(\Theta\otimes\Theta)$ is a gradient.

Now,  computing  the $L^4$-norm of the above equality,  using  the decay estimates for the Stokes semigroup \eqref{S3},
the H\"older inequality, the assumption on $\al$, and estimate \eqref{KY:dec} we obtain 
\begin{align*}
\nlp4{  u(t)-\al \Theta(t)}
&\leq \nlp4{S(t)u_0-\al\Theta(t)}\\
&\qquad +C\int_0^t(t-s)^{-\frac34}\nlp2{[(u-\al\Theta)\otimes u +\al\Theta\otimes (u-\al\Theta)](s)}\,ds\\
&\leq \nlp4{S(t)u_0-\Theta(t)}\\
&\qquad +C\int_0^t(t-s)^{-\frac34}\nlp4{(u-\al\Theta)(s)}(\nlp4{u(s)}+\nlp4{\al\Theta(s)})\,ds\\
&\leq \nlp4{S(t)u_0-\al\Theta(t)}+C\ep\int_0^t(t-s)^{-\frac34}s^{-\frac14}\nlp4{(u-\al\Theta)(s)}\,ds.
\end{align*}
Hence, denoting 
$%\begin{equation*}
\zeta(t)=t^{\frac14} \nlp4{  u(t)-\al\Theta(t)} 
$%\end{equation*}
 \ we infer that
\begin{align*}
  \zeta(t)
&\leq t^{\frac14}\nlp4{S(t)u_0-\al\Theta(t)}+C\ep t^{\frac14}\int_0^t(t-s)^{-\frac34}s^{-\frac12}\zeta(s)\,ds\\
&\leq t^{\frac14}\nlp4{S(t)u_0-\al \Theta(t)}+C\ep\int_0^1(1-\tau)^{-\frac34}\tau^{-\frac12}\zeta(t\tau)\,d\tau.
\end{align*}
Now, we compute  $\limsup\limits_{t\to\infty}$ of both sides of this inequality and we 
use \eqref{limsupx} for $p=4$. By the Lebesgue dominated convergence theorem,  we obtain
\begin{equation*}
 \limsup_{t\to\infty}\zeta(t)
\leq  C\ep\limsup_{t\to\infty}\zeta(t)\int_0^1(1-\tau)^{-\frac34}\tau^{-\frac12}\,d\tau=C_1\ep\limsup_{t\to\infty}\zeta(t),
\end{equation*}
where $C_1$ is a constant independent of $\ep$.
If  $C_1\ep<1$,  this inequality  implies immediately that  $ \limsup\limits_{t\to\infty}\zeta(t)=0$,
which is the relation   \eqref{limsup} for $p=4$. 

The same argument as above works for $p\neq4$ but the constant $C_1$ will depend on $p$ so the smallness condition $C_1\ep<1$ cannot hold true unless $\ep=0$. To get around this difficulty, we show that if  \eqref{limsup} holds true for $p=4$ then it holds true for all $p\in (2,\infty)$. Using similar computations as above we obtain
\begin{align*}
\nlp p{u(t)-\al \Theta(t)}
&\leq \nlp p{S(t)u_0-\al \Theta(t)}\\
&\qquad +C(p)\ep\int_0^t(t-s)^{-1+\frac1p}s^{-\frac14}\nlp4{u(s)-\al \Theta(s)}\,ds\\
&\leq \nlp p{S(t)u_0-\al \Theta(t)}+C(p)\ep t^{\frac1p-\frac12}\int_0^1(1-\tau)^{-1+\frac1p}\tau^{-\frac12}\zeta(\tau t)\,d\tau.
\end{align*}
Multiplying both sides of this inequality by $t^{\frac12-\frac1p}$, computing   $\limsup\limits_{t\to\infty}$, and using the 
already-proved decay for $p=4$ completes the proof of Proposition \ref{reduction-linear}.
\end{proof}

\subsection{Proof of Theorem \ref{mainthm}.}

We fix $\ep>0$ required  in Proposition \ref{reduction-linear}  and choose $\al_0\in(-\ep,\ep)$ and $T_0$ as in Corollary \ref{corol} to have that $\|u(T_0)\|_{L^{2,\infty}(\Omega)}\leq\ep$.
Let us observe  that $u(T_0)$ verifies 
\begin{equation}\label{ut0}
u(T_0)-\al\H\in \L2s. 
\end{equation}
Indeed, it follows  from Lemma \ref{l4} that $u(T_0)-\tu(T_0)-\al S(T_0)\H\in \L2s$. 
Clearly $\tu(T_0)\in \L2s$, because  $\tu$ is a square integrable weak solution of the Navier-Stokes equations.
 Moreover, we have $S(t)\H-\H\in \L2s$ as was shown in   \cite{MR2244381}. Thus, the proof of \eqref{ut0} is complete. 

In particular, using Corollary \ref{cor:dec} we have 
\begin{equation*}
\lim_{t\to\infty}t^{\frac12-\frac1p}\nlp p{S(t)(u(T_0)-\al\H)}=0\qquad\text{for every}\quad  p\in(2,\infty).
\end{equation*}
Thus, we infer from Theorem \ref{thm:oseen} that
$$
\lim_{t\to\infty}t^{\frac12-\frac1p}\nlp p{S(t)u(T_0)-\al\Theta(t)}=0\qquad\text{for every}\quad  p\in(2,\infty).
$$
Apply now Proposition \ref{reduction-linear} starting from time $T_0$ to obtain 
\begin{equation*}
\lim_{t\to\infty}(t-T_0)^{\frac12-\frac1p}  \nlp{p}{u(t)-\al \Theta(t-T_0)}=0 \qquad \text{for every}\quad p\in (2,\infty).
\end{equation*}

A calculation using the explicit formula for $\Theta$ given in \eqref{Oseen} shows that
\begin{equation*}
(t-T_0)^{\frac12-\frac1p}  \nlp{p}{\Theta(t)- \Theta(t-T_0)}  
=\nlp{p}{\Theta(1)-\Theta\bigl(\frac t{t-T_0}\bigr)}\to0\quad\text{as }t\to\infty
\end{equation*}
by the dominated convergence theorem (observe that $\bigl|\Theta\bigl(\frac t{t-T_0}\bigr)\bigr|\leq |\Theta(1)\bigr|$). This completes the proof of Theorem \ref{mainthm}.

\bigskip

{\bf Acknowledgments.}
The work of G.~Karch was partially supported 
by the MNiSzW grant No.~N~N201 418839 and 
the Foundation for Polish Science operated within the
Innovative Economy Operational Programme 2007-2013 funded by European
Regional Development Fund (Ph.D. Programme: Mathematical
Methods in Natural Sciences). The third author is partially supported by the Agence Nationale de la Recherche, Project MathOc\'ean, grant ANR-08-BLAN-0301-01.

\adrese

\end{document}